\documentclass[10pt,a4paper]{amsart}%

\usepackage{MyCommA}
\usepackage[final]{pdfpages}

\newcommand{\simpAr}[2][r]{%
\ar@{}[#1]|-*[@]_{#2}%
}
\renewcommand{\colorlinks}{true}
\renewcommand{\linkcolor}{lblue}
\renewcommand{\citecolor}{lblue}
\renewcommand{\urlcolor}{dblue}
\renewcommand{\linkbordercolor}{red}
\renewcommand{\citebordercolor}{green}
\renewcommand{\urlbordercolor}{cyan}
\hypersetup{colorlinks=\colorlinks,linkbordercolor=\linkbordercolor, linkcolor=\linkcolor, citecolor=\citecolor, urlcolor=\urlcolor,urlbordercolor=\urlbordercolor,citebordercolor=\citebordercolor}%

\def\VF{\mathrm{VF}}

\newcommand{\RF}{{\rm RF}}

\def\WF{\mathrm{WF}}

\renewcommand{\O}{\mathcal{O}}

\newcommand{\NextVer}[1]{}

\newcommand{\cCexp}{\cC^{\mathrm{exp}}}
\newcommand{\cCexpL}{\cC^{\mathrm{exp}}_\cL}

\newcommand{\Loc}{{\mathrm{Loc}}}

\def\Supp{\operatorname{Supp}}

\def\cLan{{\cL_{\rm an}^F}}

\def\11{{\mathbf 1}}
\def\AA{{\mathbb A}}

\def\CC{{\mathbb C}}

\def\FF{{\mathbb F}}

\def\NN{{\mathbb N}}

\def\QQ{{\mathbb Q}}
\def\RR{{\mathbb R}}

\def\ZZ{{\mathbb Z}}

\def\cC{{\mathscr C}}

\def\cF{{\mathcal F}}

\def\cL{{\mathcal L}}
\def\cM{{\mathcal M}}

\def\cO{{\mathcal O}}

\def\cS{{\mathcal S}}

\def\llp{\mathopen{(\!(}}

\def\rrp{\mathopen{)\!)}}

\theoremstyle{plain}

\numberwithin{equation}{subsection}

  {\par\medskip\noindent #1\par\begingroup%
    \advance\leftskip by 1em\advance\rightskip by 1em}%
  {\par\endgroup}

\DeclareMathOperator*{\dims}{dimSupp}

\newcommand{\ord}{\operatorname{ord}}

\usepackage{color}

\usepackage{amsmath}

\newcommand{\zerodel}{.\kern-\nulldelimiterspace}

\begin{document}

\author
{Avraham Aizenbud}
\address{Avraham Aizenbud,
Faculty of Mathematical Sciences,
Weizmann Institute of Science,
Rehovot, Israel}
\email{aizenr@gmail.com}
\urladdr{http://aizenbud.org}

\author
{Raf Cluckers}

\address{CNRS, Univ.~Lille,  UMR 8524 - Laboratoire Paul Painlev\'e, F-59000 Lille, France, and
KU Leuven, Department of Mathematics, B-3001 Leu\-ven, Bel\-gium}
\email{Raf.Cluckers@univ-lille.fr}
\urladdr{http://rcluckers.perso.math.cnrs.fr/}

\date{\today}

\subjclass[2010]{Primary 14E18; Secondary 22E50, 03C10, 11S80, 11U09}

\keywords{Distributions for non-archimedean local fields, WF-holonomic distributions, wave front sets, distributions of $Cexp$-class, micro-local analysis, micro-locally smooth locus, motivic constructible exponential functions, loci of motivic exponential class, embedded resolution of definable functions, uniform $p$-adic integration, Fourier transforms, $p$-adic continuous wavelet transform, Denef-Pas cell decomposition, Transfer principles for motivic integrals, alterations, $D$-modules}

%
%
%
%
%
%
%
%

\title[WF-holonomicity of $p$-adic distributions]{WF-holonomicity of ${\mathscr {C}}^{\mathrm{exp}}$-class distributions on non-archimedean local fields}
\begin{abstract}
In the context of geometry and analysis on non-archimedean local fields, we study two recent notions,  $\cCexp$-class distributions from \cite{CHLR} and WF-holonomicity from \cite{AizDr}, and we show that any distribution of $\cCexp$-class is WF-holonomic. Thus we answer a question from \cite{AizDr} by providing a framework of WF-holonomic distributions for non-archimedean local fields which is stable under taking Fourier transforms and which contains many natural distributions, in particular, the distributions studied in \cite{AizDr}. We further show that one can regularize distributions without leaving the $\cCexp$-class. Finally, we show a close link between zero loci and smooth loci for functions and distributions of $\cCexp$-class, by proving a converse to a result of \cite{CHLR}. A key ingredient is a new resolution result for subanalytic functions (by alterations), based on embedded resolution for analytic functions and model theory.
\end{abstract}

\maketitle

\section{Introduction}

\subsection{}
A key missing tool in the combination of $p$-adic geometry and analysis, is the interplay between differentiation and integration. In more advanced wordings, techniques related to
Bernstein polynomials and $D$-modules, which in the reals give a plentitude of results, seem to break down when one wants to use them in a $p$-adic context. In the real and complex setting, one thinks for example of the strong link between eigenvalues of monodromy with zeros of Bernstein polynomials and with poles of certain real integrals,  a link which remains elusive in the $p$-adic setting, see e.g.~the monodromy conjecture \cite[Chapter 1, Section 3.4]{Cham-Nica-Seba}, \cite{DenefBour}, \cite{Igusa:intro}. One also thinks of the link of $D$-modules on $\RR^n$ with distributions on $\RR^n$ and how their holonomicity is preserved under Fourier transform.  In this paper we investigate a notion of holonomicity in the $p$-adic setting, not for $D$-modules but directly for distributions on $p$-adic analytic manifolds, and its behavior under Fourier transform. Two recent contributions in this domain are combined: the notion of WF-holonomicity of \cite{AizDr} for distributions, and, the notion of distributions of $\cCexp$-class on (definable) $p$-adic manifolds from \cite{CHLR}. Our main result addresses a question from \cite{AizDr} and provides a framework of WF-holonomic distributions on $\QQ_p^n$ which is stable under Fourier transform  (and which contains the distributions studied in \cite{AizDr}), see Theorem \ref{thm:main} for the holonomicity and \cite[Theorem 3.3.5]{CHLR} for the stability under Fourier transform.

\subsection{}
The notion of  WF-holonomicity of a distribution $\xi$, introduced in \cite{AizDr}, is based on the wave front set $\WF(\xi)$ of $\xi$. So to say, the nicest distributions are the ones which come from integrating the product of a test function with a smooth (that is, $C^\infty$) density function. The wave front set of $\xi$ on a manifold $X$ is a subset of the co-tangent bundle of $X$, and it sees how far away $\xi$ is from a smooth density function, roughly by looking at decay when nearing infinity and after Fourier transformation. If the wave front set $\WF(\xi)$ is small enough, then $\xi$ is called WF-holonomic.
In more detail, if $\WF(\xi)$ is contained in a finite union of co-normal bundles of submanifolds of $X$, then $\xi$ is called WF-holonomic, see definition \ref{def:WF-hol}.

\subsection{}
In \cite{CHLR} the notion of distributions of $\cCexp$-class on $p$-adic manifolds is introduced. Stability of this class is shown under operations like Fourier transforms, pull-backs\footnote{The definition of Heifetz \cite{Hei85a} for pull-backs is made precise in \cite{CHLR} by specifying topologies on distributions in relation to their wave front sets, similar to the real case in \cite{Hormander71}. This specification also applies to Proposition 2.3.10 of \cite{AizDr}.}, and push-forwards.
In this paper we show stability in a new sense, namely under regularization: any distribution of $\cCexp$-class on $U$ can be regularized to a distribution on $X$ which is still of $\cCexp$-class, where $U$ is open in the $p$-adic manifold $X$, see Theorem \ref{thm:reg:F:one}.
Not only is the $\cCexp$-class of distributions stable under all these  operations, this class contains many natural distributions, like the ones studied in \cite{AizDr}. Distributions of $\cCexp$-class have a geometric flavour 
as reflected by the main result of this paper on  WF-holonomicity, and by their definition based on model theory.

\subsection{}
The notion of distributions of $\cCexp$-class is tightly linked to the notion of $\cCexp$-class functions, grosso modo via the continuous wavelet transform. For functions of $\cCexp$-class, the zero loci have played important roles in transfer principles (to change the characteristic of the local field) and in the description of geometric and analytic objects, see \cite{CGH} \cite{CGH2} \cite{CGH5} \cite{CHLR}.  In  \cite{CHLR} it is shown that the smooth locus (and even the micro-locally smooth locus) of a distribution of $\cCexp$-class equals the zero locus of a function of $\cCexp$-class.
We show the following converse: for any zero locus $Z(g)$ of a $\cCexp$-class function $g$ on a manifold $X$ such that $Z(g)$ is moreover open and dense in $X$, there is a distribution of $\cCexp$-class whose smooth locus equals $Z(g)$, see Theorem \ref{prop:conv}.  Together with the results \cite[Theorems 3.4.1, 4.1.2]{CHLR}, this exhibits yet another complete role played by zero loci of $\cCexp$-class functions. An analogous converse for the micro-locally smooth locus remains for the future to be discovered.


\subsection{}
A key ingredient in the proofs consists of a (new variant of a) resolution result for definable functions which creates monomials times units, but which allows finite fibers, see Theorem \ref{thm:res:def}. However, this variant does not directly allow a reduction to a Cartesian product situation when proving the WF-holonomicity from Theorem \ref{thm:main}. Indeed, an additive character evaluated in a unit times a a quotient of monomials is not at all a Cartesian product situation. However, via general properties of distributions and their wave front sets, we manage to proceed by induction on the dimension. By the finite fibers, our resolution maps are similar to alterations and remind of smoothing of real subanalytic sets as in \cite{BiersParu}.

\subsection{}
Typically, the results of this paper and of \cite{AizDr}, \cite{CHLR} hold uniformly throughout all $p$-adic fields, and, in (definable) family settings.
For simplicity of notation, a large part of the paper will be formulated for a fixed non-archimedean local field $F$ which is often assumed to be of characteristic zero, and a fixed nontrivial additive character $\psi$ on $F$. We will also state uniformity in the local field (including local fields of positive but large characteristic) and family aspects of most of our results. See for example Theorem \ref{thm:reg:fam} for a  family variant of the regularisation result. Very recently, some results of \cite{Hei85a} and \cite{CHLR} are presented in a motivic framework instead of (uniform) $p$-adic, see \cite{Raib-motivic}.
\subsection{}
The motivation for this line of research lies in part in the search for $p$-adic analogues for the strong interplay between real analysis and real geometry. Another challenge came from more global geometrical aspects than usually dealt with in model theory.
Indeed, wave front sets for distributions on a manifold $X$ cannot be seen by working piecewise on $X$ (unless the pieces are clopen), while most results on definable sets and functions are piecewise in nature. This represented the challenge to this project to combine a global geometric with a definable viewpoint. 

\subsection{Structure of the paper}
In Section \S\ref{s:hol} 
we recall the relevant terminogy and formulate the main result of the paper when working over fixed local field of characteristic $0$. In \S\ref{s:skech} we explain the main ingredients of all the proofs in the paper, with all the essential parts and omitting technical details.
In \S \ref{s:reg}-\S \ref{sec:loci} we proof all the result for fixed local filed of characteristic $0$. In \S \ref{sec:unif} we explane how to deduce uniform versions of those results when we vary the local field and allow it to have positive (but high enough) characteristic.

\subsection*{Acknowledgments}
\hspace{0.5cm}
This project was conceived while both authors participated in the Fourth International Workshop on Zeta Functions in Algebra and Geometry. We thank the organizers of the conference for creating this opportunity. R.C.~also thanks I. Halupczok and M.~Raibaut for interesting discussions on the topics of this paper.

A.A. was partially supported by ISF grant 687/13, ISF grant 249/17, and a Minerva foundation grant. R.C. was partially supported by the European Research Council under the European Community's Seventh Framework Programme (FP7/2007-2013) with ERC Grant Agreement nr. 615722
MOTMELSUM and KU Leuven IF C14/17/083, and thanks the Labex CEMPI  (ANR-11-LABX-0007-01).

\section{Wave front sets and holonomicity: definitions and main results}\label{s:hol}

\subsection{}\label{sec:all:F}
Let $F$ be any non-archimedean local field, namely, a finite field extension of $\QQ_p$ or of $\FF_p\llp t \rrp$ for some prime $p$.
Let $\cO_F$ denote the valuation ring of $F$ with maximal ideal $\cM_F$ and residue field $k_F$ with $q_F$ elements and characteristic $p_F$. We identify the value group of $F$ with $\ZZ$. We write $\ord$ for the valuation map, and $|x|$ for the norm of $x\in F$, defined as $q_F^{-\ord x}$ for nonzero $x$ and $|0|=0$.

Let $X$ be an analytic submanifold of $F^n$ of dimension $m$ for some $n\geq m\geq 0$,
see \cite{Bour} where this is called $F$-analytic instead of analytic, and see \cite[Section 2.1]{CHLR} for the more general notion of strict $C^1$ submanifolds. As in \cite[Section 2.1]{CHLR}, we will always assume that our analytic manifolds are smooth, nonempty, and of pure dimension, meaning of constant local dimension $m$ for some integer $m\geq 0$.
Define the tangent bundle $TX$ and the co-tangent $T^*X$ bundle of $X$ as usual (see \cite[Section 2.1]{CHLR}). That is, $T^*X $ is the co-tangent bundle which at $x\in X$ is the dual of the tangent space to $X$ at $x$. By the wave front set $\WF(\xi)$ of a distribution $\xi$ on $X$, we mean 
the $F^\times$-wave front set in $T^*X\smallsetminus X\times\{0\}$ in the sense of \cite[2.8.6]{CHLR} (based on Heifetz \cite{Hei85a}). (Note that Definition A.0.1 of \cite{AizDr} of wave front sets is slightly different since it includes the zero section on the support of $\xi$; this is a harmless difference.)  Let us recall these definitions. By a Schwartz-Bruhat function is meant a $\CC$-valued locally constant function with compact support. The $\CC$-vector space of Schwartz-Bruhat functions on $X$ is denoted by $\cS(X)$. A distribution on $X$ is nothing else than a linear function $\cS(X)\to\CC$ (as usual in the $p$-adic case, there are no topological requirements).  We write $\cS^*(X)$ for the $\CC$-vector space of distributions on $X$. The support of a distribution is meant as usual, see e.g.~\cite[Definition 2.4.3 ]{CHLR}.

\begin{defn}[Wave front sets]\label{Hei:smooth}
Let $U\subset F^n$ be open and let $\xi$ be a distribution on $U$. Let $(x_0,y_0)$ be in $T^*U \smallsetminus U\times\{0\}= U\times (F^n\smallsetminus\{0\})$.  Say that $\xi$ is {micro-locally smooth} at $(x_0,y_0)$ if there are open neighborhoods $U_0$ of $x_0$ and $V_0$ of $y_0$ such that for any Schwartz-Bruhat function $\varphi$ with support contained in $U_0$ there is an integer $N$ such that for all $\lambda\in F^\times$ with $|\lambda| > N$ one has
\begin{equation}\label{Hei:p}
\cF(\varphi \xi) (\lambda\cdot y) = 0 \mbox{ for all $y$ in $V_0$}.
\end{equation}
Here,
the Fourier transform $\cF(\varphi \xi)$  of the product of $\varphi$ with $\xi$ is a continuous function (see Theorem 2.5.2 of \cite{CHLR}) and $\lambda\cdot y$ stands for $(\lambda y_1,\ldots,\lambda y_n)$.
The wave front of $\xi$ is defined as the complement in $T^*U \smallsetminus U\times\{0\}$ of the set of micro-locally smooth points of $\xi$ and is denoted by $\WF(\xi)$.
More generally, for $\xi$ a distribution on an analytic submanifold $X$ of $F^n$ of dimension $m$, the wave front set of $\xi$ is defined by using analytic charts on $X$ (see  \cite[2.8.6]{CHLR} with $\Lambda=F^\times$).
\end{defn}

Define the co-normal bundle
$$
CN_Y^X \subset T^*(X)
$$
of an 
analytic submanifold $Y\subset X$ as usual, see \cite[Section 2.1]{CHLR}. That is, the co-normal bundle $CN_Y^X$ is the dual bundle of the normal bundle $N_Y^X$, where $N_Y^X$ at $y\in Y$ is the quotient of the tangent space to $X$ at $y$ by the tangent space to $Y$ at $y$.  Note that the set $CN_Y^X$ is an isotropic and even Lagrangian analytic submanifold of (the symplectic manifold) $T^*X$.

The following notions of WF-holonomicity come from \cite[Section 3.2]{AizDr}, with a slight generalization of \cite[Section 3.2]{AizDr} and \cite[Definition 4.1.1]{CHLR} (by relaxing the smoothness conditions in the algebraic case).

\begin{defn}[$\WF$-holonomicity]\label{def:WF-hol}
Let $X$ be an analytic submanifold of $F^n$ and let $\xi$ be a distribution on $X$.
Say that  $\xi$ is analytically
WF-holonomic if the wave front set of $\xi$ is contained in
a finite union of co-normal bundles $CN_{Y_i}^X$ where each $Y_i\subset X$
is an analytic submanifold of $X$.
Call $\xi$ algebraically WF-holonomic if moreover one has $\dim X=\dim \overline{X}^{\rm Zar}$ and  $\dim Y_i=\dim  \overline{Y_i}^{\rm Zar}$, with $\overline{X}^{\rm Zar}$ and the $\overline{Y_i}^{\rm Zar}$ the Zariski closure of $X$ and the $Y_i$.
\\

If $X$ and the $Y_i$ are mereley strict $C^1$ submanifolds (instead of analytic), then we say strict $C^1$ WF-holonomic for the corresponding notion.
\end{defn}

\begin{remark}\label{rem:var:hol}
Note that we don't assume smoothness of $\overline{X}^{\rm Zar}$ (neither of the $\overline{Y_i}^{\rm Zar}$), and that $F$ is allowed to have positive characteristic. In these ways, the above definition of algebraic WF-holonomicity generalizes the notion of algebraic WF-holonomicity of \cite[Definition 3.2.1]{AizDr} and \cite[Definition 4.1.1]{CHLR}. In the case that $\overline{X}^{\rm Zar}$ is smooth and that $F$ has zero characteristic, all these definitions of algebraic WF-holonomicity coincide (note that the smoothness condition is forgotten in the final part of Proposition 4.3.1 of \cite{CHLR}, see Remark \ref{rem:amend:4.3.1} below).
 \end{remark}



\subsection{Definitions and results for a fixed local field $F$}\label{sec:fixed:F}
\label{subsec:F-X}
\label{subs:F}

From  now on, and until the end of Section \ref{sec:loci}, we fix a local field $F$ of characteristic zero and an additive chacacter $\psi:F\to\CC^\times$ which is trivial on $\cM_F$ and nontrivial on $\cO_F$. (Thus, $F$ is a finite field extension of $\QQ_p$ for some prime $p$.)  (An additive character is a continuous group homomorphism from the additive group on $F$ to $\CC^\times$.)

The advantage of working with fixed $F$ is the ease of presentation both for definable sets and for the rings of complex valued functions that we integrate (called functions of $\cCexp$-class). Uniformity will come at the end of the paper, in Section \ref{sec:unif}.

\subsection{Languages on $F$: subanalytic and semi-algebraic}

For each integer $n\geq 0$, let $\cO_F\langle x_1,\ldots,x_n\rangle$ be the $p$-adic completion of $\cO_F[x_1,\ldots,x_n]$ inside $\cO_F[[x_1,\ldots,x_n]]$, for the Gauss-norm.
Note that $\cO_F\langle x_1,\ldots,x_n\rangle$ consists of power series $\sum_{i\in\NN^n} a_i x^i$ in multi-index notation and with $a_i\in \cO_F$ such that  $|a_i|$ goes to zero when $|i|:=i_1+\ldots + i_n$ goes to infinity.  For $f$ in $\cO_F\langle x_1,\ldots,x_n\rangle$, write $\bar f$  for the restricted analytic function associated to $f$, namely, the function $\bar f:F^n\to F$ sending $z\in \cO_F^n$ to the evaluation $f(z)$ of $f$ at $z$ (i.e. the $p$-adic limit of the partial sums), and sending the remaining $z$ to $0$.

In this section \ref{sec:fixed:F} we use the following two languages (in the first order sense of model theory). Let $\cL^F$ be the ring language (namely having symbols $+,-,\cdot,0,1$), together with constant symbols from $\cO_F$. Let $\cLan$ be $\cL^F$ together with for each $f \in \cO_F\langle x_1,\ldots, x_n\rangle$ a function symbol for the restricted analytic function $\bar f$ associated to $f$. A set $X\subset F^n$ is called $\cL^F$-definable, resp.~$\cLan$-definable, if there is an $\cL^F$-formula, resp.~an $\cLan$-formula $\varphi(x)$ with free variables $x_1,\dots,x_n$, such that $X$ consists of the values for $x=(x_1,\dots,x_n)$ that make $\varphi$ valid in $F$. A function between definable sets is called definable if its graph is a definable set. (All this is standard in model theory and first order logic.) These definable sets and functions are called semi-algebraic, resp.~subanalytic for $\cL^F$, resp.~$\cLan$, and have many geometric properties, enabled in the first place by quantifier elimination results in closely related (slightly bigger) languages, see e.g.~\cite{Mac}, \cite[Theorem 5.6]{Prestel}, \cite{Denef2}, \cite{DvdD}, \cite{vdDHM}.

The dimension of a nonempty $\cLan$-definable set $C\subset F^n$ is defined as the maximum of the dimensions of analytic submanifolds of $F^n$ contained in $C$ (see Section 3.15 of \cite{DvdD} and Lemma \ref{lem:DDset} for the dimension theory of $\cL$-definable sets). If moreover $C$ is  $\cL^F$-definable, then it is equal to the dimension of the Zariski closure of $C$, see \cite{Dries-dim} or \cite{SvdD}.

\subsection{$\cCexpL$-class and $\cL$-WF-holonomicity}

From now (and until the end of Section \ref{sec:loci}) we fix $\cL$ to be either $\cL^F$ or $\cLan$. By an $\cL$-manifold we mean an $\cL$-definable set $X\subset F^n$ for some $n\geq 0$ such that $X$ is moreover an analytic submanifold of $F^n$. By an $\cL$-analytic map we mean an $\cL$-definable, analytic map between $\cL$-manifolds. By an  analytic isomorphism we mean an analytic bijection between analytic manifolds whose inverse is also analytic.
We now come to our key definitions.

\begin{defn}[Functions of $\cCexpL$-class]\label{defn:cexp}
Let $X\subset F^n$ be an $\cL$-definable set. The $\CC$-algebra $\cCexpL(X)$ is defined as the sub-$\CC$-algebra of all complex valued functions on $X$ generated by functions $X\to \CC$ of the following forms:
\begin{enumerate}
\item\label{gen:1} $x\mapsto |f(x)|$,  

\item\label{gen:2} $x\mapsto \ord g(x)$, 

\item\label{gen:3} $x\mapsto \psi(h(x))$, 
\end{enumerate}
where $f:X\to F$, $h:X\to F$, and $g:X\to F^\times$ are $\cL$-definable functions and where $\psi$ is the additive character fixed above.
A function in $\cCexpL(X)$ 
is called a function of $\cCexpL$-class.
\end{defn}

These algebras of Definition \ref{defn:cexp} are versatile because of their stability under integration (and thus under Fourier transforms), see Section 8.6 of \cite{CLexp} and Theorem 3.2.1 of \cite{CGH}, and, they inherit geometrical properties from their definable building blocks, see e.g.~\cite{CGH5}.

For $x\in F^n$ and $r\in\ZZ$, write $B_r(x)$ for the ball $\{y\in F^n\mid \ord (y-x)\geq r\}$, where the order of a tuple is the minimum of the orders of the entries. Write $\11_A$ for the characteristic function of a subset $A\subset S$ (where the superset $S$ is usually implicitly clear).

\begin{defn}[Distributions of $\cCexpL$-class]\label{defn:cexp-d}
We say that a distribution $\xi$ on an $\cL$-submanifold $X\subset F^n$ is of class $\cCexpL$ if the following condition on the continuous wavelet transform of $\xi$ is satisfied: The function
$$
D_\xi:X\times F^\times \to\CC
$$
is a $\cCexp$-class function, where
$$
D_\xi(x,r) =  \begin{cases}
   \xi(\11_{B_{|r|}(x)\cap X}) \mbox{  if $B_{|r|}(x)\cap X$ is compact},\\
0 \mbox{ otherwise.}
\end{cases}
$$
We call $D_\xi$ the $B$-function of $\xi$ (where the letter $B$ comes from ball).
\end{defn}

Proposition \ref{prop:Julia} below implies that the condition for a distribution $\xi$ on $X$ to be of $\cCexpL$-class is independent of the embedding of $X$ into $F^n$.  Note that $D_\xi$ is a continuous wavelet transform of $\xi$.

\begin{defn}[$\cL$-WF-holonomicity]\label{defn:L-WF}
Let $\xi$ be a distribution of class $\cCexpL$ on an $\cL$-manifold $X$. Then $\xi$ is called $\cL$-WF-holonomic if
if the wave front set of $\xi$ is contained in
a finite union of co-normal bundles $CN_{Y_i}^X$ where each $Y_i\subset X$
is an $\cL$-submanifold of $X$.
\end{defn}

\subsection{Main results for fixed $F$}

Now we can formulate our main results. The holonomicity result is in fact the key goal of this paper.
 (Recall that $F$ has characteristic zero from Section \ref{sec:fixed:F}  until the end of Section \ref{sec:loci}.)

\begin{thm}[Holonomicity]\label{thm:main}
Let $X$ be an $\cL$-manifold and let $\xi$ be a distribution on $X$ of class $\cCexpL$.
Then $\xi$ is $\cL$-WF-holonomic.

In particular, $\xi$ is analytically WF-holonomic, and if $\cL$ is $\cL^F$, then $\xi$ is algebraically WF-holonomic. 
\end{thm}

The following regularization result, allowing to extend distributions to larger domains without leaving the $\cCexp_\cL$-class, will be useful to show Theorem \ref{thm:main}.

\begin{theorem}[Regularisation]\label{thm:reg:F:one}
Consider an $\cL$-manifold $X$.
Let $U$ be a nonempty $\cL$-definable open subset of $X$.
Then the restriction map
$$
\cS^* (X) \to \cS^*(U)
$$
admits a linear section $\cS^* (U) \to \cS^*(X)$
 that maps  distributions of $\cCexp_\cL$-class  to distributions of $\cCexp_\cL$-class.
\end{theorem}
With $U$ and $X$ as in the theorem, for any $\xi$ in $\cS^* (U)$ and any linear section $\kappa:\cS^* (U) \to \cS^*(X)$, $\kappa(\xi)$ is usually called a regularization of $\xi$.

In \cite{CHLR}, Theorem 3.4.1 (resp.~Remark 4.3.3), it is shown that the wave front set of a $\cCexpL$-class distribution equals the complement of the zero locus of a function of the same class. It would be interesting to find a precise criterion for such zero loci so that they are the complement of a wave front of some $\cCexpL$-class distribution; we give a partial answer to this question in Theorem \ref{prop:conv}.
The smooth locus of a distribution $\xi$ on a analytic submanifold $X\subset F^n$ is defined as the set of those $x \in X$ which allow an open neighborhood $U$ such that the restriction of $\xi$ to $U$ is a smooth measure.
We know that the smooth locus of a $\cCexpL$-class distribution $\xi$ on $X\subset F^n$ is dense open in $X$ by Theorem 4.1.2 
of \cite{CHLR}, see Theorem \ref{thm:gen.sm} below.
The following result gives the converse to these properties.

\begin{theorem}[Correspondence of loci]\label{prop:conv}
Let $X$ be an $\cL$-manifold and let $g$ be in $\cCexpL(X)$.
If the zero locus $Z(g)$ of $g$ is dense open in $X$, then there exists a distribution $\xi$ on $X$ which is of $\cCexpL$-class and such that the smooth locus of $\xi$ equals $Z(g)$.
\end{theorem}

The result shows that zero loci of $\cCexpL$-class functions are the right objects to describe smooth loci of $\cCexpL$-class. (For other objects described precisely by loci of $\cCexpL$-class functions, see e.g.~\cite{CGH}, \cite{CGH5}.)

The following is our resolution result for $\cLan$-definable functions on $\cLan$-definable sets, refining Theorems 2.2 and 2.4 of \cite{DvdD}. Note that the resolving maps $\varphi_i$ have restrictions to $U = (\cO_F\smallsetminus\{0\})^m $ which are only locally isomorphisms.

\begin{defn}\label{def:unit:mon}
An analytic function $u:\cO_F^m\to \cO_F$ is called an analytic unit on $\cO_F^m$ if it is given by a power series which is an invertible element in the ring $\cO_F\langle x_1,\ldots,x_m \rangle$. By a monomial on $\cO_F^m$ we mean a function $M: \cO_F^m \to \cO_F$ sending $x$ to $d\cdot\prod_{i=1}^m x_i^{\mu_i}$ with exponents $\mu_i$ which are natural numbers and with $d\in \cO_F$.
\end{defn}
\begin{theorem}[Resolution result for  $\cLan$-definable sets and functions]\label{thm:res:def}
Let $X\subset \cO_F^n$ be a closed $\cLan$-definable set of pure dimension $m$ and let $f:X\to \cO_F^k$ be an $\cLan$-definable function for some $k$ and $n\geq m\geq 0$. Write $U$ for $(\cO_F\smallsetminus\{0\})^m$.
Then there exist finitely many
$\cLan$-definable functions
$$
\phi_i: \cO_F^m\to X
$$
such that each $\phi_i$ is analytic, proper, and the following hold for some positive integers $d_i$.
\begin{enumerate}
\item\label{thm:res:def:1}
The set $U_i:=\phi_i(U)$ is open in $X$.
\item\label{thm:res:def:2} The restriction $\phi_i|_{U}$ is a local isomorphism to $U_i$ with finite fibers of fixed size $d_i$.
\item\label{thm:res:def:3} There are analytic units $u_{ij}$ on $\cO_F^m$ and monomials $M_{ij}$ on $\cO_F^m$
such that for each component $f_j$ of $f$ one has
\begin{equation}\label{fjvarphii}
f_j (\phi_i (x)) = u_{ij}(x) M_{ij}(x) \mbox{ for each $x$ in $U$ and each $i,j$.}
\end{equation}
\item\label{thm:res:def:4} The $U_i$ are pairwise disjoint, and, the union of the $U_i$ is dense open in $X$.
\item\label{thm:res:def:5}  For each $i$ one has either $u_{i1}=1$, or, $M_{i1}=1$.
\end{enumerate}
\end{theorem}

The resolution theorem can of course also be applied to other situations, e.g. to $F$-valued functions $f_j:X\to F$ instead of $\cO_F$-valued, by working on pieces where $|f_j|\leq 1$, resp.~where $|f_j|> 1$ and replacing $f_j$ by $1/f_j$ on the latter.

%
%

We will give uniform versions of Theorems \ref{thm:main}, \ref{thm:reg:F:one} and \ref{prop:conv} below in Section \ref{sec:unif}.

\section{Skech of the proofs}\label{s:skech}

We start with proving regularization (Theorem \ref{thm:reg:F:one}).  By partition of unity, the question is local, so it is enough to extend a $\cCexpL$-class distribution for an open definable subset $U$ to a larger open definable subset $X$ in $F^n$. We can stratify the complement $Z:=X\smallsetminus U$ to definable manifolds. Proceeding by induction we can assume that $Z$ is smooth. Again using locality of the question and a suitable version of the implicit function theorem we can assume that  $Z$ is a graph. In this case, we can extend our distribution using a definable tubular neighborhood of $Z$.

The proof of holonomicity (Theorem \ref{thm:main}) is based on a theorem from \cite{CHLR} stating that any $\cCexpL$-class distribution is smooth on an open dense subset. We use a partition of unity and regularisation to reduce to proving WF-holonomicity of distributions $\xi$ on $X=\cO_F^n$. The next ingredient in the proof is a key lemma (Lemma \ref{lem:key}) stating that any smooth $\cCexpL$ function $f$ on an open dense definable set $U \subset X$ can be extended to a $\cCexpL$-class WF-holonomic distribution on $X$ (maybe after restricting it first to a smaller open dense subset). The Key lemma and the result from  \cite{CHLR} allow us to replace $\xi$ with a distribution $\xi'$ whose support is of smaller dimension.  We resolve (using Theorem \ref{thm:res:def}) the support of $\xi'$. We then use regularisation in order to construct a distribution on the resolution which coincides with $\xi'$ on an open dense set. Using the push forward of that distribution, and the induction assumption, we can replace $\xi'$ with another distribution whose support has even smaller dimension. We continue by induction until we kill $\xi$ completely.

The proof of the Key lemma is based on resolution of singularities for definable functions (Theorem \ref{thm:res:def}). Using it we can reduce to the case that $U=(\cO_F\smallsetminus \{0\})^n$ and $f$ has an explicit form containing (quotients of) monomials, units, the absolute value, the valuation, and the additive character $\psi$. In this case we explicitly construct a $\cCexpL$-class distribution on $X$ that extends $f$ and prove that it is WF-holonomic.

\begin{remark}
$ $
\begin{itemize}
\item Both the holonomicity theorem and the key lemma are proven by induction. However, the Key lemma is less suitable for induction since it only claims \textbf{existence} of an extension with certain properties. Therefore when we prove the Key lemma it is more convenient to use the holonomicity theorem for smaller dimension rather than the Key Lemma itself. Because of that, we prove both results together by induction.
\item  The proof of the regularization result works uniformly for the algebraic language and in the analytic one. However, this is not the case for the Holonomicity theorem. The proof of the Holonomicity theorem relies on resolution of singularities for definable functions, which does not (yet) work so well in the algebraic language. Though resolution of definable sets in the algebraic language is not a problem (modulo Hironaka's theorem), resolution of definable functions seems to be hard to deduce from the existing literature. Therefore we chose to prove the holonomicity for the analytic language first, and then deduce it for the algebraic language based on \cite{CHLR}.
\item The resolution result is, in fact, an alteration rather than a modification since it is only a \textbf{local} isomorphism on an open dense set. An actual modification cannot resolve a root function to a monomial function.
However, it is not a problem, since our use of resolution is for pushing forward distributions, so we just have to divide by the size of the fiber (which we assured to be constant) to get the desired result.
\item Although the resolution result is crucial for the Key-lemma it does not resolve it compliantly, and some additional analysis is required. The reason is that even after the resolution the explicit form of the function $f$ is not a product of functions that each depend on only one coordinate. This is because composition of an additive character with a quotient of monomials is not such function.
\end{itemize}
\end{remark}

The proof of the resolution of singularities of definable sets and functions (Theorem \ref{thm:res:def}) is based on Hironaka's theorem and results from \cite{vdDHM} on decompositions of definable functions to simpler functions called terms (in the sense of model theory), of a suitable language. We first use Hironaka's theorem to resolve terms on $\cO_F^n$, similar as in \cite{DvdD}. We then use this in order to resolve closed definable subsets of full dimension in $\cO_F^n$. This is possible since such subsets can be defined (quantifier free) by terms. This is the first place where we actually need alterations and not just modifications, since definable sets can be of the form $\{x|\exists y \text{ such that  } x=y^n\}$ which can not be resolved just by modification. We next resolve a general definable function $f$ on a closed definable set $X\subset \cO_F^n$ of full dimension. For this we use the results from \cite{vdDHM} and  alternating the following two procedures that are possible because of the previous steps:
\begin{itemize}
\item Decompose $X$ to into definable subsets (of full dimension) and deal with (the closure of) each one separately.
\item Resolve any term that we need in order to resolve $f$.
\end{itemize}
Finally we do the general case by decomposing any definable set to graphs over definable sets of full dimension. Note that two kinds of terms are used: one without root functions (which suffice for quantifier elimination), and one with root functions (in which definable functions become piecewise terms). Some extra work is done to reduce the resolution of terms in the richer language to terms in the smaller language (without roots). Also in this reduction we need alterations and not just modifications.

In order to prove that  any zero locus of a $\cCexp$ function  can be a smooth locus of  a $\cCexp$ distribution (Theorem \ref{prop:conv}), we first show that a  zero locus of a $\cCexp$ function is also a  zero locus of a bounded $\cCexp$ function. Then we prove the theorem using the following statements:
\begin{itemize}
\item for any $\cCexp$ function $g$ on $F^n$ there is a stratification of $F^n$ by manifolds s.t. $g$ is smooth on each strata.
\item any submanifold $X$ of $F^n$ has a canonical measure with full support on $X$.
\end{itemize}

The uniform versions of Theorems \ref{thm:reg:F:one} and \ref{prop:conv} are proved in the same way as the theorems themselves. We deduce the uniform version of  Theorem \ref{thm:main} from results of  \cite{CHLR} about the uniform nature  of the Wave front of a $\cCexp$ distribution and  Theorem \ref{thm:main} itself.

\section{Proof of the regularization}\label{s:reg}

To prove Theorem \ref{thm:reg:F:one} we will need to work with $\cCexp_\cL$-families of Schwartz Bruhat functions, which we now define, and which combines well with distributions of $\cCexp_\cL$-class by Proposition \ref{prop:Julia}.

\begin{defn}
\label{def:fam:one}
Consider an $\cL$-manifold $X$ and let $Y$ be an $\cL$-definable set. A family $(\varphi_y)_{y\in Y}$ of functions $\varphi_y$ in $\cS(X)$ is called a $\cCexp_\cL$-family when the function
$$
(x,y)\in X\times Y \mapsto \varphi_y(x)
$$
lies in  $\cCexp_\cL(X\times Y)$.
\end{defn}

Theorem \ref{thm:reg:F:one} will follow from the following results.

\begin{prop}[\cite{CHLR}] 
\label{prop:Julia}
Consider an $\cL$-manifold $X$ and let $Y$ be an $\cL$-definable set. Let $\xi$ be a distribution on $X$ of $\cCexp_\cL$-class and let  $(\varphi_y)_{y\in Y}$ be a $\cCexp_\cL$-family of functions $\varphi_y$ in $\cS(X)$ for some definable set $Y$.
Then the function  sending $y\in Y$ to $\xi(\varphi_y)$ is of $\cCexp_\cL$-class.
 \end{prop}
\begin{proof}
This follows from Proposition 3.3.4 and Remark 4.3.3 of \cite{CHLR}.
\end{proof}

\begin{lemma}[Definable Urysohn's Lemma]\label{lem:urysohn}
Let $X\subset F^n$ be an $\cL$-manifold and consider $\cL$-definable sets $Z\subset U\subset X$ such that $U$ is open in $X$ and $Z$ is closed in $X$. Then there exists an $\cL$-definable clopen $C\subset X$ such that $Z\subset C\subset U$. (A clopen set is a set which is open and closed.)
\end{lemma}
\begin{proof}
For any $x \in Z$ let $B_x$ be the maximal ball around $x$ satisfying
\begin{itemize}
\item $B_x$ is of radius $\leq 1$,
\item $B_x\cap X$ is compact,
\item $B_x \cap X\subset U$.
\end{itemize}
Note that such a ball $B_x$ exists for every $x\in Z$. Put
$$
C=\bigcup_{x\in Z} B_x \cap X.
$$
We obviously have $Z\subset C\subset U$ and that $C$ an $\cL$-definable open subset of $U$. It remains to prove that $C$ is closed in $X$.
Let $\alpha_i\in C$ be  such that $(\alpha_i)_{i\in\NN}$ is a converging sequence with limit $\alpha$ in $X$.  It is sufficient to show that $\alpha$ lies inside $C$. To this end, it is enough to find a converging subsequence with limit in $C$.  Let $z_i\in Z$ such that  $\alpha_i \in B_{z_i}$  for each $i$.
There are two cases two consider.
\begin{enumerate}[{Case} 1:]
\item the $B_{z_i}$ become identical to each other when $i$ is large enough.\\
In this case we can assume that all $\alpha_i$ are in one ball $B_{z_{i_0}}\cap X$ and the statement follows from the compactness of  $B_{z_{i_0}}\cap X$ and the fact that $B_{z_{i_0}}\cap X\subset C$.
\item Up to passing to a subsequence, the $B_{z_i}$ are pair-wise different.\\
Up to replacing by the subsequence and by the ultrametric, the $B_{z_i}$ are pairwise disjoint. On the other hand, $\alpha_i$ forms a Cauchy sequence. Thus, the radius of $B_{z_i}$ converges to $0$ when $i$ grows. This implies that $\lim z_i=\alpha$. Since $Z$ is closed in $X$, this implies that $\alpha\in Z$.
\end{enumerate}
\end{proof}

\begin{prop}[Partition of unity]\label{prop:part}
Let $X\subset F^n$ be an $\cL$-manifold and let $X=\bigcup_{i=1}^N U_i$ be a finite cover with $\cL$-definable open subsets $U_i$ of $X$. Then there exists a finite cover $X=\bigcup_{j=1}^{N'} U_j'$ with disjoint $\cL$-definable clopen sets refining the cover $X = \bigcup_{j=1}^{N}  U_i$. (Refining means that for any $j$, the set $U'_j$ is contained in some $U_i$.)
\end{prop}
\begin{proof}
For each $i=1,\ldots,N$, define $U_i'$ recursively to be an $\cL$-definable clopen  set given by the previous lemma (Lemma \ref{lem:urysohn}) such that
 $$\big(X \smallsetminus\bigcup_{j < i} U'_j \big) \smallsetminus\bigcup_{j > i} U_j  \subset U'_i \subset U_i.$$
This gives the desired refinement with $N'=N$.
\end{proof}

The following refines the results \cite[Theorem 1.1]{SvdD}, \cite[Theorem 3.14]{DvdD}, \cite[Proposition 1.5.3]{CCL} in the sense that our covers are moreover open.

\begin{lemma}\label{lem:imp}
Let $Y\subset F^n$ be an  $\cL$-submanifold. Then one can find a finite cover $Y=\bigcup U_i$ by open $\cL$-definable sets   $U_i$ such that each $U_i$ is a graph of an analytic $\cL$-definable function from an open subset $V_i \subset L_i$ of a linear subspace of $F^n$ of dimension $\dim Y$ to a linear complement of this subspace.
\end{lemma}
\begin{proof}
Since the case $\dim Y=n$ is obvious we will assume that  $\dim Y<n$.
For each $y\in Y$ there exist a coordinate subspace $L$ such that  the differential of the  projection $p$ from $Y$ to $L$ at the point $y$ is an analytic isomorphism and hence $p:Y\to L$ is a local analytic isomorphism around $y$ by the analytic inverse function theorem.  Without loss of generality we can pass to an open cover and assume that the same subspace $L$ can be used for all points $y\in Y$.
By existence of $\cL$-definable sections and since the cardinality of the fibers of $p:Y\to L$ is bounded (both statements follow e.g.~from the cell decomposition theorems from \cite{Denef2} \cite{Ccell}), we can find finitely many definable (not necessarily continuous) sections $s_i:p(Y)\to Y$ such that the images of the $s_i$ cover $Y$.

By \cite[Theorem 1.1]{SvdD}, \cite[Theorem 3.14]{DvdD} (or \cite[Proposition 1.5.3]{CCL}), we can partition $p(Y)=\bigcup_{j} S_{j}$ into finitely many $\cL$-manifolds such that for each $i,j$ the restriction $s_i|_{S_{j}}$ is analytic and such that $S_{j}$ is the graph of an $\cL$-analytic function from an open $W_{j}\subset L_{j}$ to $L_{j}'$ where $L=L_{j} \oplus L_{j}'$ as linear spaces.
It suffices to show that, up to refining the partition $\bigcup_{j} S_{j}$, one can extend $s_i|_{S_{j}}$ analytically to an open $\cL$-definable neighborhood $V_{ij}$ of $S_{j}$ in $L$.


To this end, fix $i$ and $j$, and, for any $x \in S_{j}$, let $B_{x,i}$ be the maximal ball in $F^n$ around $s_i(x)$ and of radius at most $1$ such that $p|_{B_{x,i} \cap Y}$ is injective. Let $\nu_{x,i}$ be the inverse of $p|_{B_{x,i} \cap Y}: B_{x,i} \cap Y \to p(B_{x,i} \cap Y)$ and  put $V_{ijx}:=p(B_{x,i} \cap Y)  \cap (x+L'_{j})$. 
We now obtain an extension of  $s_{i}|_{S_{j}}$ to the open $\cL$-definable set $V_{ij} := \bigcup_{x\in S_{j} }  V_{ijx}$ as needed, namely, sending $x+y$ in $V_{ij}$ with $x\in S_{j}$ and $y\in L'_{j}$ to $\nu_{x,i}(x+y)$.  
\end{proof}

\begin{prop}\label{thm:sec:F:one}
Consider an $\cL$-manifold $X$.
Let $U$ be a nonempty $\cL$-definable open subset of $X$, let $\xi$ be a distribution on $U$ of $\cCexp_\cL$-class and write $Z$ for the complement of $U$ inside $X$.
Suppose that $Z$ is an $\cL$-submanifold of $X$.
Then the restriction map
$$
\cS(X)\to \cS(Z)
$$
admits a linear section
$$
\nu: \cS(Z)\to \cS(X)
$$
that maps $\cCexp_\cL$-class families to $\cCexp_\cL$-class families. (Namely, if  $(\varphi_y)_{y\in Y}$ is a  $\cCexp_\cL$-class family of functions $\varphi_y$ in $\cS(Z)$, then  $(\nu(\varphi_y))_{y\in Y}$ is a  $\cCexp_\cL$-class family of functions $\varphi_y$ in $\cS(X)$.)
\end{prop}

\begin{proof}
Let $n$ be such that $X\subset F^n$.
$ $
\begin{enumerate}[{Case} 1:]
\item $X$ is open in $F^n$ and $Z$ is a graph of a map $\phi$ from an open $V\subset F^{k}$ to  $F^{n-k}$.\\
For any $z\in V$ let $B_z$ be the  maximal ball in $F^n$ of radius at most $1$ around $(z,\phi(z))$ that is contained in $X$. Let $p:F^n\to F^k$ be the coordinate projection so that $p(Z)=V$.
For a function $f \in \Sc(Z)$ and $x \in X$ define
$$\nu(f)(x):=f(p(x), \phi(p(x)))\cdot \11_{B_{p(x)}}(x).$$
It is easy to see that $\nu$ is  a section as desired.
\item $X$ is open in $F^n$.\\
By Lemma \ref{lem:imp} we can find a finite $\cL$-definable open cover $Z=\bigcup U_i$ such that each  $U_i$ is a graph of an analytic map from an open subset of  $L_i$ to  $L_i'$ where $F^n=L_i\oplus L_i'$ is a decomposition to vector spaces.
We can find $\cL$-definable open sets $V_i \subset X$ such that $U_i=V_i\cap Z$. Adding the set $X\smallsetminus Z$ we obtain a finite open cover of $X$. Applying partition of unity (Proposition \ref{prop:part}) to this cover we reduce to the previous step.
\item General case.\\
By Lemma \ref{lem:imp} again, we can cover $X$ by finitely many open $\cL$-sets each of which is isomorphic  to an open $\cL$-definable subset of $F^k$ with $k$ the dimension of $X$. Again Applying partition of unity (Proposition \ref{prop:part}) to this cover we reduce to the previous step.
\end{enumerate}
\end{proof}

\begin{proof}[Proof of Theorem \ref{thm:reg:F:one}]
Let $Z$ be the complement of $U$ in $X$.
\begin{enumerate}[{Case} 1.]

\item $Z$ is a closed $\cL$-submanifold of $X$. \\
Let $\nu$ be the section obtained from Proposition \ref{thm:sec:F:one}.  To $\varphi$ in $\cS(X)$ we associate $\widetilde \varphi$ in $\cS(U)$ by defining $\widetilde \varphi$  as the restriction of $\varphi-\nu(\varphi|_{Z})$ to $U$.
Now define the section $\kappa$ of $\cS^* (X) \to \cS^*(U)$  by sending $\xi$ in $\cS^* (U)$ to the distribution $\kappa(\xi) = \xi_X$ where $\xi_X(\varphi)$ for any $\varphi$ in  $\cS(X)$  is defined as $\xi(\widetilde \varphi)$. That $\xi_X$ is a distribution of class  $\cCexp_\cL$  now follows from Proposition \ref{prop:Julia} and the fact $\nu$ maps $\cCexp_\cL$-class families to $\cCexp_\cL$-class families. Linearity is clear by construction.
\\  

\item General case.\\
By Theorem \cite[Theorem 4.2.5]{CCL} (more concisely, by the combination of the frontier condition for stratifications and \cite[Theorem 1.1]{SvdD}, \cite[Theorem 3.14]{DvdD} as showed in particular in \cite{CCL}),
there exist definable sets $Z_i$ for $i=0,\ldots, k $ for some $k>0$ such that $Z=Z_0\supset Z_1 \supset \cdots \supset Z_n=\emptyset$ and such that $Z_{i}\smallsetminus Z_{i+1}$ is a closed $\cL$-submanifold of $Z_i$. Put $U_i:=X \smallsetminus Z_i$. By composition it is enough to prove that   $\cS^*(U_{i+1})\to \cS^*(U_i)$ admits a section $\kappa_i$ as desired, but this follows from the previous case.
\end{enumerate}
\end{proof}

\section{Proof of resolution for definable sets and functions}\label{sec:resolution}
We will prove our variant of resolution of singularities for $\cLan$-definable functions.
This is a resolution by alterations rather then by modifications
in the sense that finite fibers are allowed.
For $\cLan$-definable functions which are moreover compositions of converging power series and restricted division, similar resolution results have been obtained in \cite{DvdD}, using Hironaka's embedded resolution of singularities from  \cite{Hir:Res}. Similar to Hironaka's result, one of our main purposes is to make the pull-back of the definable function a product of a monomial with a unit.
Since in the $p$-adic case, fragments of $r$-th root functions are definable for integers $r>0$, we will need to combine power maps and monomialization to get our resolution result. 
For us, after monomialization, the unit will still have an important role, since it can in general not be neglected inside the argument of the additive character $\psi$. Again because of the argument inside $\psi$, to prove Theorem \ref{thm:main} will require additional work, even  after resolving singularities, and will not reduce directly to a Cartesian product situation.

Write $\cL_{\rm an, qe}^{F}$ for $\cLan$ together with a function symbol for field division sending nonzero $x$ to $x^{-1}$ and $0$ to $0$, and relation symbols $P_n$ for each $n>0$ for the set of nonzero $n$-th powers in $F$.
Similarly, write $\cL_{\rm an, D}^{F}$ for the language $\cLan$ together with a function symbol for restricted division $D$ sending $(x,y)\in\cO_F^2$ to $x/y$ when $|x|\leq |y|\not=0$ and to zero otherwise, and relation symbols $P_n$ for each $n>0$ for the set of nonzero $n$-th powers in $\cO_F$. By the variant from \cite{vdDHM} of the quantifier elimination result from \cite{DvdD}, the structure $F$, resp.~$\cO_F$, has quantifier elimination in the language $\cL_{\rm an, qe}^{F}$, resp.~$\cL_{\rm an, D}^{F}$.

We will derive Theorem \ref{thm:res:def} from the resolution results (2.2) and (2.4) from \cite{DvdD}, and the piecewise description of definable functions by terms in a slightly larger language. Recall that a term in a language is  a finite composition of function symbols from the language. Similarly as for quantifier elimination, one needs an adapted language  to ensure that definable functions are piecewise equal to terms (and, to ensure that definable sets are given by quantifier free formulas). See Definition \ref{def:unit:mon} for the notions of analytic units and monomials on $\cO_F^m$.


\begin{prop}[\cite{DvdD}]\label{prop:Hir}
Let $f:\cO_F^m\to \cO_F^k$ be an analytic map. Then there exists a compact analytic manifold $C$ and an analytic map $h:C\to \cO_F^m$ such that $h$ is an analytic isomorphism on the preimage of a dense open subset of $\cO_F^m$ and such that
for each $c\in C$ there is an open neighborhood $O_c$ of $c$ and an analytic isomorphism $\varphi_c:\cO_F^m\to O_c$ such that 
for each $i=1,\ldots,k$,
$$
f_i\circ h\circ \varphi_c = u_{ic}M_{ic}
$$
where $u_{ic}$ is an analytic unit on $\cO_F^m$ and $M_{ic}$ is a monomial on $\cO_F^m$.
\end{prop}
\begin{proof}
The case that $F=\QQ_p$ is the resolution result (Theorem 2.2) from \cite{DvdD}.
The same proof applies to any finite field extension $F$ of $\QQ_p$, based on the excellence result (Theorem 1.2) of \cite{GrothDieu} and Hironaka's results from \cite{Hir:Res}, as explained in \cite{DvdD}.
\end{proof}

\begin{prop}[\cite{DvdD}]\label{prop:unif}
Let $f:\cO_F^m\to \cO_F^k$ be a map whose component functions are $\cL_{\rm an, D}^{F}$-terms, for some $m>0$ and $k>0$.  Then there exists a compact analytic manifold $C$ and an analytic map $h:C\to \cO_F^m$ such that $f\circ h$ is analytic.
\end{prop}
\begin{proof}
In \cite{DvdD}, proof of (2.4), it is shown (but not stated) that $f\circ h$ can be supposed to be analytic on $C$ when $F=\QQ_p$. The same proof applies for any finite field extension $F$ of $\QQ_p$.
\end{proof}

The following result relates $\cLan$-definable functions to terms, at the cost of taking powers.
\begin{prop}[\cite{vdDHM}]\label{prop:thm5.5}
Let $f_i:X\subset F^n\to F$ be finitely many $\cLan$-definable functions. Then there is an integer $M>0$ and there are finitely many disjoint definable subsets $X_s$ of $X$ and finitely many definable functions $f_{i s \ell}$ on $X_s$ such that
$$
f_{is \ell }^{M}
$$
is given by an $\cL_{\rm an, qe}^F$-term for each $i$, $s$,  $\ell$ and such that
$$
\sum\limits_\ell f_{is \ell } = f_i  \mbox{ on $X_s$}.
$$
\end{prop}
\begin{proof}
This follows from Theorem 5.5 of \cite{vdDHM} and its proof. 
\end{proof} 

We will use Proposition \ref{prop:thm5.5} only once, namely in Case 3 of the proof of Theorem \ref{thm:res:def}. One can also treat that case in an alternative way and avoid the use of Proposition \ref{prop:thm5.5} by proceeding instead by induction on the complexity of terms based on the weaker result (than Proposition \ref{prop:thm5.5}) that $\cLan$-definable functions are piecewise equal to terms in an expansion of $\cL_{\rm an, qe}^F$-terms with root-functions, see \cite[Theorem 7.5]{CLR}.

The following (basic) lemma states that $\cL_{\rm an, qe}^{F}$-terms correspond piecewise to
$\cL_{\rm an, D}^{F}$-terms. The lemma after this one gives a similar result for definable sets.


\begin{lem}\label{lem:DDterm}
Consider a tuple $t(x)$ of  $\cL_{\rm an, qe}^{F}$-terms in $n$ variables $x_i$.
Let $X \subset \cO_F^n$ be an $\cL_{\rm an, qe}^{F}$-definable set.
Assume that
$t(x)\in \cO_F^s$ for each $x\in X$. Then there exist finitely many tuples $t_j$ of $\cL_{\rm an, D}^{F}$-terms such that for each $x\in X$ one has
$$
t(x) = t_{j}(x)
$$
for some $j$.
\end{lem}





\begin{proof}
Since a term by definition is a composition of function symbols in the language, any $\cL_{\rm an, qe}^{F}$-term can be  written as composition of functions $f_i:F^l\to F^k$ such that each of the $f_i$ satisfies one of the following:
\begin{enumerate}
\item\label{it:term.1} all the components of $f_i$ are rational functions of the form $p\cdot (q)^{-1}$ for some polynomials $p$ and $q$,
\item\label{it:term.2} all the components of $f_i$ are restricted analytic functions.
\end{enumerate}
Since a composition of rational functions as in (\ref{it:term.1}) is piecewise equal to a rational function as in (\ref{it:term.1}) (with definable pieces), we may suppose that
$$
t=a_1\circ b_1 \circ \cdots \circ a_N\circ b_N \circ  a_{N+1}
$$ where the $a_i$  satisfy (\ref{it:term.2}) and the $b_i$ satisfy (\ref{it:term.1}).
For any rational function $r$ of the form $p\cdot (q)^{-1}$ for some polynomials $p$ and $q$ write $r_D$ for the function $D(p,q)$, where $D$ stands for restricted division. Likewise, for a tuple $r$ of such rational functions $r_i$ write $r_D$ for the corresponding tuple of the $r_{i,D}$.
Note that $a_1\circ b_1 (z) $ equals $a_1 \circ b_{1,D}(z)$ for any $z$ in $\cO_F^\beta$ with $\beta$ the arity of $b_1$. The lemma now follows easily by induction on $N$, by focusing on the remaining part $a_2\circ b_2 \circ \cdots \circ a_N\circ b_N \circ  a_{N+1}$. %
\end{proof}

\begin{lem}\label{lem:DDset}
Any $\cL_{\rm an, qe}^{F}$-definable set $X\subset \cO_F^n$ is $\cL_{\rm an, D}^{F}$-definable and vice versa.
\end{lem}
\begin{proof}
Clearly any $\cL_{\rm an, D}^{F}$-definable set in $\cO_F^n$ is also $\cL_{\rm an, qe}^{F}$-definable. For the other direction consider the $\cL_{\rm an, qe}^{F}$-definable embedding $i:F\to \cO_F^2$ given by

$$
i(x)=
\begin{cases}
(x,0) & \text{ if } x\in \cO_F ,\\
(x^{-1},1)  & \text{ otherwise.}
\end{cases}
$$
The assertion now  follows from
the following simple observations:
\begin{itemize}
\item There is an  $\cL_{\rm an, D}^{F}$-definable function $a:\cO_F^4\to \cO_F^4$ such that  for any $x,y\in F$ we have $a(i(x),i(y))=i(x+y)$.

\item There is an  $\cL_{\rm an, D}^{F}$-definable function $m:\cO_F^4\to \cO_F^4$ such that  for any $x,y\in F$ we have $m(i(x),i(y))=i(xy)$.

\item for any restricted analytic $f$ function on $F$, there exist an  $\cL_{\rm an, D}^{F}$-definable function $A_f:\cO_F^2\to \cO_F^2$ such that  for any $x\in F$ we have $A_f(i(x))=i(f(x))$.

\item There is an  $\cL_{\rm an, D}^{F}$-definable function $inv:\cO_F^2\to \cO_F^2$ such that  for any $x\in F$ we have $inv(i(x))=i(x^{-1})$ (recall that $x^{-1}$ is interpreted as $0$ if $x=0$).
\item For any integer $n$ There is a  $\cL_{\rm an, D}^{F}$-definable set $\Pi_n\subset \cO_F^2$ such that  for any $x\in F$ we have $i(x)\in\Pi_n$ if and only if $x\in P_n$.
\end{itemize}
\end{proof}

Finally, in order to prove Part \eqref{thm:res:def:5} of the resolution theorem we will use the following standard result.
\begin{lemma}\label{lem:swal.un} Let $u$ be an analytic unit on $\O_F^n$ and $M$ be a monomial on $\O_F^n$. Consider the function $f=uM$ on $\O_F^n$. Then there exist a finite disjoint $\cL$-definable open cover $\O_F^n=\bigcup_{i=1}^N U_i$ and $\cL$-definable analytic isomorphisms  $\phi_i:\O_F^n\to U_i$ such that
\begin{itemize}
\item $\phi_i^*(f)$ is either a monomial on $\O_F^n$, or, a constant times an analytic unit on $\O_F^n$,
\item for any analytic unit $u'$ on $\cO_F^n$ and any monomial $M'$ on $\cO_F^n$, the function $\phi_i^*(u'M')$ equals the product of  a monomial with an analytic unit on $\cO_F^n$.
\end{itemize}
\end{lemma}
\begin{proof}
Let us first treat the special case that $M(x_1,\dots,x_n)=x_1^{k_1}\cdots x_n^{k_n}$ for some natural numbers $k_i$, that $k_1\geq 1$, and, that there is an analytic unit $v$ on $\cO_F^n$ such that $v^{k_1}=u$ and such that $x_1v(x)$ is a special restricted power series in the sense of \cite[Section 2.2]{Igusa:intro}.
Consider the map $\psi:\cO_F^n\to \O_F^n$ given by
$$\psi(x_1,\dots, x_n)=(x_1v(x_1,\dots,x_n),x_2,\dots,x_n).$$
By \cite[Corollary 2.2.1]{Igusa:intro}, the map $\psi$ is an analytic isomorphism, say, with inverse $\phi$. Then $\phi_1=\phi$ is as desired (with $N=1$), in particular, one has $\phi^*(f)=M$.

Let us now reduce to the conditions of the special case. Fix $a=(a_1,\dots,a_n) \in \O_F^n$. Since any two balls in $\O_F^n$ are either disjoint or contained in one another and by compactness of $\O_F^n$, it is enough to find a ball $B_a$ around $a$ and an $\cL$-definable analytic isomorphism $\phi_a :\O_F^n \to B_a$ satisfying both conditions of the lemma.
We may assume that there is $i$ with $a_i= 0$ and that the monomial $M$ depends nontrivially on the coordinate $x_i$. Indeed,  otherwise we can take a small enough ball $B_a$ around $a$ and $\phi$ to be a homothety after a translation to make $\phi^*(f)$ an analytic unit times a constant. Without loss of generality we may assume that $i=1$. Also we can assume that $a_1=\cdots a_k=0$ for some $k\geq 1$ and $a_{k+1},\cdots,a_n\neq 0$.
Write $$M(x_1,\dots,x_n)=x_1^{k_1}\cdots x_n^{k_n}.$$ Let $L:=(x_1,\dots,x_n):=x_{1}^{k_{1}}\cdots x_m^{k_m}$ and $K(x_1,\dots,x_n):=x_{m+1}^{k_{m+1}}\cdots x_n^{k_n}.$ We have $M=KL$. Note that $K$ is a unit in a small ball $B'$ around  $a$.
Consider the affine transformation $t$ which is a homothety after a translation and which maps $0$ to $a$ and $\O_F^n$ onto $B'$. Up to choosing $B'$ small enough, the pull-back $t^*(f)$ is of the form as in our special case, and, the lemma now follows from this special case applied to $t^*(f)$.
\end{proof}

We can now give the complete proof of Theorem \ref{thm:res:def}.

\begin{proof}[Proof of Theorem \ref{thm:res:def}]
It is enough to prove only statements \eqref{thm:res:def:1} -- \eqref{thm:res:def:4} since statement \eqref{thm:res:def:5} can be deduced from them by lemma \ref{lem:swal.un} (and refining the cover $U_i$ and changing the maps $\phi_i$).
$ $
\begin{enumerate}[{Case} 1:]

\item $X=\cO_F^n$ and the $f_i$ are $\cL_{\rm an, D}^F$-terms.

This case follows from Propositions \ref{prop:Hir} and \ref{prop:unif}. Indeed, the disjointness of the $\phi_i(U)$ is easily obtained on top of the conclusions of \ref{prop:Hir} and \ref{prop:unif}.
%

\item $X$ is of dimension $m=n$ and $f$ is constant.

By Lemma \ref{lem:DDset} and quantifier elimination in $\cL_{\rm an, D}^F$, the set $X$ is defined by a finite Boolean combination of condition of the form $t_i(x)\in P_{n_i}$
for some $\cL_{\rm an, D}^F$-terms $t_i$ and some $n_i\geq 1$.  (Note that a condition $t=0$ corresponds to $t$ not being in $P_1$.) Let $N$ be the product of all occurring $n_i$.

By the previous case we can assume that
the terms $t_i$
are analytic units times monomials.
 Without loss of generality we can suppose that the occurring units have constant coset in $F^\times$ modulo $P_N$.
 For any $\lambda=(\lambda_1,\dots, \lambda_n)\in (\O_F\smallsetminus\{0\})^n$ define $\phi_\lambda:\O_F^n\to \O_F^n$ by $\phi_\lambda(x_1,\dots, x_n)=(\lambda_1x_1^N,\dots, \lambda_n x_n^N)$.
 We can find finitely many $\lambda^1,\dots \lambda^K\in \O_F^n$ such that  the images of $\phi_{\lambda^j}$ cover the entire $\O_F^n$ and such that the sets $\phi_{\lambda^j}(U)$ are disjoint.
 Thus $\dim \O_F^n \smallsetminus \bigcup_j \phi_{\lambda^j}(U)<n$. This implies that $\dim X \smallsetminus \bigcup_j \phi_{\lambda^j}(U)<n=\dim X$ and thus (since $X$ has pure dimension) $ \bigcup_j \phi_{\lambda^j}(U)\cap X$ is dense in $X$. Note that for any $j$, either $\phi_{\lambda^j}(U) \subset X$ or  $\phi_{\lambda^j}(U) \cap X=\emptyset$. We obtained that the collection $$\{\phi_{\lambda^j}|\phi_{\lambda^j}(U) \subset X\}$$ meets the requirements.

\item\label{c:ful.dim} $X$ is of dimension $m= n$ and $f$ is general.

By Lemma \ref{lem:DDterm} and Proposition \ref{prop:thm5.5},  each component function of $f$ is piecewise equal to a sum of definable roots of $\cL_{\rm an, D}^F$-terms $t_{ij}$. We may suppose that there is only one piece, let's still call it $X$.
By Case 2 we can suppose that $X=\cO_F^n$. (Indeed, apply Case 2 to $X$ and work with one chart $\phi_i$ and the pull-back of $f$ along $\phi_i$.) Now apply Case 1 to the terms $t_{ij}$. By composing with $N$-th power maps (as the $\phi_\lambda$ of Case 2) for some highly divisible $N$, we reduce to the case that the $t_{ij}$ are  $\cL_{\rm an, D}^F$-terms, and hence, also that $f$ is given by a tuple of $\cL_{\rm an, D}^F$-terms. Now we end by Case 1 for these $\cL_{\rm an, D}^F$-terms.

\item General case.

Write $\dim X=m$ as in the theorem. Take an open dense subset of $X$ which is an $\cL_{\rm an, D}^F$-manifold, which is possible by \cite[Theorem 3.14]{DvdD} and Lemma \ref{lem:DDset}.  By Proposition \ref{prop:part} and Lemma \ref{lem:imp}, we may now suppose that $X$ equals the topological closure of the graph 
of an $\cL_{\rm an, D}^F$-definable function $g:V\to \cO_F^{n-m}$ for some $\cL_{\rm an, D}^F$-definable open $V$ of $\cO_F^m$ and for some coordinates on affine space. Extend $g$  by $0$ to a function $\bar g$ on $\bar V$. Using Case 3 for $\bar g$ on $\bar V$, we can assume that $\bar V=\cO_F^m$ and that each component of $\bar g|_{(\cO_F\smallsetminus 0)^m}$ equals a monomial times an analytic unit. This means that $g$ can be extended continuously to a function $\widetilde g$ on $\cO_F^m$, and thus, $X$ equals the graph of  $\widetilde g$. This gives us an $\cL$-definable analytic isomorphism of $\cL$-manifolds $i:\cO_F^m \to X$. The assertion follows now from Case  \ref{c:ful.dim} applied to $\cO_F^m$ and $i^*(f)$.
\end{enumerate}
\end{proof}

It may be interesting to look for a definable, strict $C^1$ variant of Theorem \ref{thm:res:def} that can be shown without using analyticity and without using Hironaka's resolution results, see Remark \ref{rem:hensel} about a possible axiomatic approach in which piecewise analyticity may not hold.

\section{Proof of holonomicity}\label{sec:holonom:an}

\subsection{Proof of the analytic case}
In this section we prove Theorem \ref{thm:main} for $\cL=\cLan$.

Consider for any $n>0$ the Haar measure $|dx|$ on $F^n$ normalized so that $\cO_F^n$ has measure $1$.
By a smooth measure on $X$ we mean a distribution on $X$ which is locally (at any point $x\in X$) either zero or given by integration against the measure associated to an analytic volume form on $X$. Note that the wave front set in a way describes the non-smooth aspect of a distribution, and, in particular, the wave front set of a smooth measure is empty.
We will deduce Theorem \ref{thm:main} from the following theorem from \cite{CHLR} and our Key Lemma \ref{lem:key} by using regularization and induction on the dimension of $X$.

\begin{thm}[\cite{CHLR}]\label{thm:gen.sm}
Let $X$ be an $\cL$-submanifold of $F^n$ of dimension $m$ and let $\xi$ be a distribution on $X$.
If $\xi$ is of class $\cCexpL$ then $\xi$ is smooth when restricted to $X \smallsetminus C$, where $C\subset F^n$ is an  $\cL$-definable set of dimension less than $m$.
\end{thm}
\begin{proof}[Proof of Theorem \ref{thm:gen.sm}]
If $\xi$ is of class $\cCexp$, then this is a special case of Theorem 4.1.2 of \cite{CHLR}. For the $\cCexp_{{\rm an}}$-class, one moreover uses \cite[Remark 4.3.3]{CHLR}. 
\end{proof}

\begin{lemma}[Key Lemma]\label{lem:key}
Suppose that $\cL=\cLan$.
Let $U$ be an $\cL$-definable dense open subset of a compact $\cL$-manifold $X$. Let $\mu$ be a smooth measure of class 
$\cCexpL$ on $U$.
Then there is an $\cL$-definable dense open $V\subset U$ such that $\mu|_{V}$ can be extended to an $\cL$-WF-holonomic distribution on $X$ of class $\cCexpL$.
\end{lemma}

Based on the key lemma for $\dim X \leq n$ and assuming Theorem \ref{thm:main} for $\dim X < n$ and with $\cL=\cLan$ we can now prove Theorem \ref{thm:main} for $X$ with  $\dim X=n$ and $\cL=\cLan$. 
In section \ref{sec:key}, we will prove the Key Lemma for $\dim X=n$ assuming Theorem \ref{thm:main} for $X$ with $\dim X < n$. This will complete the proof of  Theorem \ref{thm:main}
for $\cL=\cLan$.
In \S\S \ref{ssec:Hol.alg}
we will deduce  Theorem \ref{thm:main} in the general case.

\begin{proof}[Proof of Theorem \ref{thm:main} assuming $\cL=\cLan$]
Denote by $\dims(\xi)$ the minimal dimension of an $\cL$-definable set $Y\subset X$ such that $\Supp(\xi)\subset Y$.
The proof is by induction on 
$\dims(\xi)$. The base case is trivial. 

By Lemma \ref{lem:imp} we find a finite, open  $\cL$-definable cover $X=\bigcup_i U_i$  together with open $\cL$-definable embeddings $U_i \to \O_F^n$ of $\cL$-manifolds. (Recall that $\dim X= n$.) Using partition of unity (Proposition \ref{prop:part}) we can reduce to the case that $X$ is an open $\cL$-definable set in  $\O_F^n$. Using  the regularisation result (Theorem \ref{thm:reg:F:one})  we can reduce to the case $X=\O_F^n$. We proceed by analyzing the following cases.

\begin{enumerate}[{Case} 1:]
\item  $\dims(\xi)=\dim(X)$\\
By Theorem \ref{thm:gen.sm} there is a definable open dense subset $V\subset X$ such that the restriction $\xi|_V$ is smooth.
By our Key Lemma \ref{lem:key}, and up to making $V$ smaller if necessary, we can extend $\xi|_V$ to an $\cL$-WF-holonomic distribution $\xi'$ on $X$ of class
 $\cCexpL$. Decompose $\xi$ as the sum $\xi=\xi'+(\xi-\xi')$. By the induction assumption (on $\dims$), $\xi-\xi'$ is $\cL$-WF-holonomic (indeed, its support lies in $X\setminus V$ which is of dimension less than $\dims(\xi)$ by properties of dimensions of definable sets). Since the sum of $\cL$-WF-holonomic distributions is $\cL$-WF-holonomic, the theorem follows.
\item  $\dims(\xi)<\dim(X)$\\
Let $Y\subset X$ be a closed $\cL$-definable  set such that $\Supp(\xi)\subset Y$ and $\dim(Y)=\dims(\xi)$.

We apply resolution of singularities (Theorem \ref{thm:res:def}) for the definable set $Y$ and the constant function $1$ on it. Let $U$, $\phi_i$, $U_i$ and $d_i$ be as in Theorem \ref{thm:res:def}. Let $V=\bigcup U_i$.
Put $$\xi':=\xi|_{X \smallsetminus (Y  \smallsetminus  V)}.$$
The distribution $\xi'$ is supported on $V$ and thus can be thought of as a distribution on $V$. It is easy to see that as such it is also  of $\cCexpL$-class.
Put
$\xi_i=\xi'|_{U_i}.$ Using regularisation (Theorem \ref{thm:reg:F:one}) we can extend $\phi_i^{*}(\xi_i)$  to a $\cCexpL$-class distribution $\xi'_i$ on $\cO_F^{\dim Y}$. Each $\xi'_i$  is $\cL$-WF-holonomic by the inductive hypothesis (indeed, $\dim Y<n$). Let $$\xi'':=\sum (\phi_i)_*(\xi'_i)/d_i.$$ Note that $\xi''|_V=\xi|_V$.




By \cite[Proposition 3.2.7 (2)]{AizDr} (see also  \cite[Theorem 2.9]{Hei85a} and \cite[Theorem 2.9.6]{CHLR}) on push-forwards, $\xi''$ is also $\cL$-WF-holonomic. Furthermore, $\xi''$ is of $\cCexpL$-class by the stability under push-forward from \cite[Theorem 3.4.5 and Remark 4.3.3]{CHLR}. Decompose $$\xi=\xi''+(\xi-\xi'').$$ Again we are done by induction on $\dims(\xi)$. Indeed, $(\xi-\xi'')$ has a lower dimensional support than $\xi$.
\end{enumerate}
\end{proof}


\subsubsection{Proof of the Key lemma}\label{sec:key}

Define a regular triple to be a triple $(\mu,U,X)$ that satisfies the assumptions of the Key lemma, namely, $\mu$ is a $\cCexpL$-class smooth measure on an $\cL$-definable dense open $U$ of the compact $\cL$-manifold $X$ and where $\cL=\cLan$.
Call such a regular triple good if the Key lemma holds for it. Precisely, a regular triple $(\mu,U,X)$ is called good if there is an $\cL$-definable dense open $V$ of $U$  such that  $\mu|_V$ has an extension to an $\cL$-WF-holonomic distribution on $X$ which is of $\cCexpL$-class.

The Key Lemma will follow from the Resolution Theorem for definable functions,  
the following straightforward proposition and lemma, and an inductive procedure in tandem with the proof of our main holonomicity result.

\begin{proposition}\label{prop:triples}
Suppose that $\cL=\cLan$.
Consider a regular triple  $(\mu,U,X)$. Then the following properties hold.
\begin{enumerate}
 \item\label{trip2} If  $(\mu_1,U,X)$ and $(\mu_2,U,X)$  are good then so is $(\mu_1+\mu_2,U,X)$.
\item\label{trip3} If  $(\mu_1,U_1,X_1)$  $(\mu_2,U_2,X_2)$  are good then so is $(\mu_1\boxtimes\mu_2,U_1\times U_2 ,X_1\times X_2)$, with $\mu_1\boxtimes\mu_2$ the product measure.
\item\label{trip4} If  $(\mu,U,X)$ is good and $f$ is a smooth (namely, locally constant) $\cCexpL$-function on $X$ then $(\mu f|_U,U,X)$ is good.
\item\label{trip7}
Assume that  $(\mu_1,U_1,X_1)$ and $(\mu_2,U_2,X_2)$ are regular triples and that $\varphi:X_1\to X_2$ is a
proper, $\cL$-analytic map such that
 \begin{itemize}
\item $\varphi(U_1)$ is open in $U_2$.  
\item $\varphi|_{U_1}$ is a local isomorphism onto $\varphi(U_1)$.   
\item $(\varphi|_{U_1})_{*}(\mu_1)=\mu_2$, that is, the push forward along $\varphi|_{U_1}$ of the distribution $\mu_1$ on $U_1$ equals  $\mu_2$ (as distributions on $U_2$).
\end{itemize}
Then, if  $(\mu_1,U_1,X_1)$ is good, then so is $(\mu_2,U_2,X_2)$.
\end{enumerate}
\end{proposition}
\begin{proof}
Only (\ref{trip7}) needs a proof (the other properties follow more easily from the corresponding properties of wave front sets and analytic manifolds).

For Property (\ref{trip7}), take an $\cL$-definable dense open $V_1$ of $U_1$ and a good extension $\xi_1$ on $X_1$ of $\mu_1|_{V_1}$ (namely, an extension which is of $\cCexpL$-class and $\cL$-WF-holonomic). The push forward $\varphi_*(\xi_1)$ is of $\cCexpL$-class and $\cL$-WF-holonomic.  Indeed, the push-forward of an analytically WF-holonomic distribution under a proper analytic map is again analytically WF-holonomic by \cite[Proposition 3.2.7 (2)]{AizDr} (see also  \cite[Theorem 2.9]{Hei85a} and \cite[Theorem 2.9.6]{CHLR}), and, the push-forward of a $\cCexpL$-class distribution under a proper $\cL$-analytic map is again of $\cCexpL$-class (see \cite[Theorem 3.4.5 and Remark 4.3.3]{CHLR}). 
Let $V_2$ be the union of $\varphi(V_1)$ with  $U_2\smallsetminus \overline{\varphi(U_1)}$, where $\overline{\varphi(U_1)}$ is the closure of $\varphi(U_1)$. The set $\varphi(U_1)$ is automatically disjoint from $\varphi(X_1\smallsetminus U_1)$.  Hence, by construction and definability properties, $V_2$ is an $\cLan$-definable dense open in $U_2$, and, $\varphi_*(\xi_1)$ extends $\mu_2|_{V_2}$. Hence, we are done.
\end{proof}

\begin{lem}\label{trip1}
Consider an $\cL$-submanifold $X\subset F^n$. Let $U\subset X$ have finite complement in $X$, and let $\xi$ be a distribution of $\cCexpL$-class on $U$ such that $\xi$ is $\cL$-WF-holonomic. Then there is a distribution $\xi_X$ on $X$ which is of $\cCexpL$-class, whose restriction to $U$ equals $\xi$, and such that $\xi_X$ is $\cL$-WF-holonomic.
\end{lem}
\begin{proof}
Regularize $\xi$ to a distribution $\xi_X$ on $X$ using a section as given by Theorem \ref{thm:reg:F:one}. Then $\xi_X$ is as desired. Indeed, $\xi_X$ is $\cL$-WF-holonomic since $U$ has finite complement in $X$.
\end{proof}

We show the following lemma assuming Theorem \ref{thm:main} for $X$ of dimension less than $n$.
\begin{lem}\label{lem:allgood}
Suppose that $\cL=\cLan$.
Let a good triple $(\mu,U,X)$  be given with  $\dim X \leq n$. Let $\xi$ be any $\cCexpL$-class distribution on $X$ which coincides with $\mu$ on a dense open $V$ of $U$. Then $\xi$ is $\cL$-WF-holonomic.
\end{lem}
\begin{proof}
Let a distribution $\xi_X$ on $X$ be given by the goodness of the triple. Write  $\xi = \xi_X + (\xi - \xi_X)$.  Then, by the argument of case 2 of the proof of Theorem \ref{thm:main} and by our assumption that Theorem \ref{thm:main} holds when $\dim X <  n$, we find that $\xi - \xi_X$ and hence also $\xi$ are $\cL$-WF-holonomic.
\end{proof}

We can now prove our Key Lemma for $\dim X=n$ assuming  the main holonomicity theorem (Theorem \ref{thm:main}) for $X$ with $\dim X < n$ and $\cL=\cLan$.

\begin{proof}[Proof of the Key Lemma \ref{lem:key}]
Let a regular triple $(\mu,U,X)$  be given.  We proceed by induction on $\dim X$,  where the one-dimensional case is taken care of by Lemma \ref{trip1}.
By the definition of $\cCexpL(U)$, $\mu$ is a finite sum of terms of the form
$$
x\in U\mapsto c \psi(f_{1}(x))  |f_2(x)|   \prod_{i=3}^k \ord (f_i(x)) ,
$$
where the $f_i$ are $\cL$-definable $F$-valued functions and where $c\in \CC$, and where $f_i\not=0$ for $i>2$. By (\ref{trip2}) of Proposition \ref{prop:triples}, and up to replacing $U$ with a dense open, we may suppose that $\mu$ equals one such term. Indeed, any such term is locally constant on a dense definable open.
By working piecewise on $X$ we may suppose for each $i$ that either $|f_i|>1$, or, $|f_i|\leq 1$ holds. Apply Theorem \ref{thm:res:def} to the function whose $i$th component is $f_i$ if $|f_i|\leq 1$ on $X$ and $1/f_i$ otherwise. By (\ref{trip2}) and (\ref{trip7}) of Proposition  \ref{prop:triples}, this reduces the case to $U$ being $(\cO_F\smallsetminus\{0\})^m$, $X=\cO_F^m$, and each $f_i$ (or $1/f_i$) being of the form (\ref{fjvarphii}).
We may thus suppose that 
$\mu$ is of the form
$$
\mu(x) = c \psi\big(u(x)M(x)^{\eta_1}\big) \prod_{i=1}^m |x_i|^{\eta_2 s_i}  \ord(x_i)^{t_i}
$$
where $c\in \CC$, $s$ and $t$ lie in $\NN^m$,   $\eta_1$ and $\eta_2$ lie in $\{1,-1\}$, 
$u$ is an analytic unit on $\cO_F^m$ and $M$  is a monomial on $\cO_F^m$.
By  \eqref{thm:res:def:5} of Theorem \ref{thm:res:def}, we may furthermore suppose
that either $M(x)=1$ or $u(x)=1$ for all $x$ in $\cO_F^m$.

In the first case that  $M(x)=1$ on $\cO_F^m$ we are done by (\ref{trip3}) and (\ref{trip4}) of Proposition \ref{prop:triples}. Indeed, $\psi(u(x))$ is smooth and nonvanishing on $\cO_F^m$, and $\mu(x)/\psi(u(x))$ is a Cartesian product situation with one-dimensional Cartesian factors each of which falls under Lemma \ref{trip1}.
Similarly one treats the case that $\eta_1=1$.
Let us now treat the final case that $u(x)$ is constant and $M(x)$ is non-constant on $\cO_F^m$, and $\eta_1=-1$.
We first regularize $\mu$ to a $\cCexpL$-class distribution $\xi$ on $\cO_F^m$.
We will then show, by working locally, that $\xi$ is $\cL$-WF-holonomic. Recall that $U$ is $(\cO_F\smallsetminus\{0\})^m$ and that $X$ is $\cO_F^m$.


Let $p$ be the function $\cO_F^m\to \cO_F^m$ which sends $x$ to $(x_1,\ldots,x_{i-1},0,x_{i+1},\ldots,x_m)$ where $i$ is the minimal number in $\{1,\ldots,m\}$ such that $|x_i|$ is minimal among the $|x_j|$ for $j$ in $\{1,\ldots,m\}$.
Let $\varphi_0$ be a locally constant function on $U^c := \cO_F^m \smallsetminus U$ (with the subset topology) and with compact support. Associate to $\varphi_0$ a function $L(\varphi_0)$ on $\cO_F^m$ by sending $x\in \cO_F^m$ to $\varphi_0(p(x))$. Clearly $L(\varphi_0)$ is a Schwartz-Bruhat function on $\cO_F^m$.

For $\varphi$ a Schwartz-Bruhat function on $\cO_F^m$, let $\varphi^c$ be the restriction of $\varphi$ to $U^c$, and consider the lift $L(\varphi^c)$. Let $\widetilde \varphi$ be the restriction to $U$ of $\varphi - L(\varphi^c)$. Then by construction $\widetilde \varphi$ is a Schwartz-Bruhat function on $U$. Moreover, the function sending $\varphi$ to $\widetilde \varphi$ is a $\CC$-linear map from $\cS(\cO_F^m)$ to $\cS(U)$, which is the identity on $\cS(U)$. Now let $\xi$ be the distribution on $\cO_F^m$ sending a Schwartz-Bruhat function $\varphi$  on $\cO_F^m$ to the evaluation of the distribution $\mu$ at $\widetilde \varphi$. Clearly $\xi$ is of $\cCexpL$-class and extends $\mu$ (use Proposition \ref{prop:Julia}). There is only left to show that $\xi$ is $\cL$-WF-holonomic, which is a local property.

Fix a point $a$ in $\cO_F^m\smallsetminus U$. It is enough to show that the restriction of $\xi$ to a small neighborhood of $a$ is $\cL$-WF-holonomic. By a reasoning as for Lemma \ref{trip1}, it is sufficient to treat the case that $a\not=0$.
Up to reordering the variables, we may suppose that the first coordinate $a_{1}$ of $a$ is nonzero. If $M(x)$ does not depend on $x_1$, then we are done by (\ref{trip3}) of Proposition \ref{prop:triples} and induction on the dimension of $X$ and Lemma \ref{lem:allgood}. If $M(x)$ depends nontrivially on $x_1$ but on no other variable, we can finish similarly. Now suppose that $M(x)$ depends nontrivially on $x_1$ and, say, also nontrivially on $x_2$.
Write $M(x)=x_1^{k_1}\cdots x_m^{k_m}$. Consider the map $\phi:\O_F^m\to \O_F^m$
$$(x_1,\dots,x_n)\mapsto (a_1 x_1^{k_2},x_2 x_1^{-k_1},x_3\dots, x_m).$$
There exists a  small ball $B$ around $(1,a_2,\dots,a_m)$ such that  $\phi|_B$ is a proper analytic isomorphism onto an open neighborhood of $a$.
So, it is enough to prove that $(\phi^{*}(\xi))|_B$ is $\cL$-WF-holonomic.

Note that $(B,\phi^{-1}(U)\cap B, \phi^{*}(\mu) )$ is a good triple by (\ref{trip3}) of Proposition \ref{prop:triples} and induction on the dimension of $X$ (indeed, the pull-back of $M$ along $\phi$ does not involve $x_1$ anymore and hence one can apply (\ref{trip3}) of \ref{prop:triples}). By Lemma \ref{lem:allgood}  we are done for $(\phi^{*}(\xi))|_B$.
\end{proof}

\subsection{Holonomicity: the algebraic case}\label{ssec:Hol.alg}

In this section we prove Theorem \ref{thm:main} for $\cL=\cL^F$.
We will first treat the case that the Zariski closure of $X$ in $\AA_F^n$ is smooth using Theorem \ref{thm:main} with $\cL=\cLan$ as proved in Section \ref{sec:holonom:an} and by Proposition 4.3.1 from \cite{CHLR}. The general case will follow from this smooth case by our partition of unity result and by reducing to graphs.

The following remark amends Proposition 4.3.1 of \cite{CHLR} by making explicit the smoothness condition.
 \begin{remark}\label{rem:amend:4.3.1}
In \cite{CHLR}, the notion of algebraic WF-holonomicity is only defined for distributions on analytic submanifolds $X\subset F^n$ such that the Zariski closure of $X$ in $\AA^n_F$ is smooth. Therefore, the condition that the Zariski closure of $W_{F,y}$ in $\AA_F^n$ is smooth should be added as an extra assumption at the start of the `Moreover' statement of Proposition 4.3.1, for each $y$ and $F$.
\end{remark}

\begin{prop}\label{thm:hol:alg}
Let $X\subset F^n$ be an $\cL^F$-manifold such that the Zariski closure of $X$ in $\AA_F^n$ is smooth.
Let $\xi$ be a distribution on $X$ of $\cCexp_{\cL^F}$-class. Suppose that $\xi$ is strict $C^1$ WF-holonomic. Then $\xi$ is $\cL^F$-WF-holonomic and thus algebraically WF-holonomic  (see Definitions \ref{def:WF-hol}, \ref{defn:L-WF}). 
\end{prop}
\begin{proof}
The result is a special form of 
Proposition 4.3.1 of \cite{CHLR} and  Remark \ref{rem:amend:4.3.1}. 
\end{proof}

\begin{proof}[Proof of Theorem \ref{thm:main} for $\cL=\cL^F$]
We may suppose that $X\subset F^n$.  If the Zariski closure of $X$ in $\AA_F^n$ is smooth, then we are done by Proposition \ref{thm:hol:alg} and by the above proved case of Theorem \ref{thm:main}  for $\cLan$ (which contains $\cL^F$). Indeed, $\cLan$-WF-holonomicity implies strict $C^1$ WF-holonomicity. Now let $X\subset F^n$ be a general $\cL^F$-manifold. By Lemma \ref{lem:imp} and Proposition \ref{prop:part}, we may suppose that $X$ is of dimension $n$. Indeed, the pieces given by Lemma \ref{lem:imp} can be taken clopen and disjoint by Proposition \ref{prop:part}. Hence, we are done by the previous case since the Zariski closure of $X$ now equals $\AA^n_F$ which is smooth. This finishes the proof of Theorem \ref{thm:main}.
\end{proof}

\section{Smooth loci and zero loci}\label{sec:loci}

A zero locus of a $\cCexpL$-class function equals the zero locus of a bounded function of the same class, as follows.

\begin{prop}\label{prop:int-loc}
Let $g$ be in $\cCexpL(X)$ for some $\cL$-definable set $X$. Then there is a function $h$ in $\cCexpL(X)$ such that $|h(x)|_\CC$ is bounded on $X$, and such that the zero locus of $h$ equals the zero locus of $g$.
\end{prop}
\begin{proof}
We will construct $h$ by multiplying $g$ with a function $f$ in $\cCexpL(X)$ such that $f$ takes positive real values at most $1$. Such a product clearly preserves the zero locus. Write $g$ as a finite sum of products of generators $T_i$ of the forms (\ref{gen:1}), (\ref{gen:2}), and (\ref{gen:3}) of Definition \ref{defn:cexp}. In each generator $T_i$ there occurs an $F$-valued definable function, say, $t_i$. For each $x\in X$, let $\alpha(x)$ be the maximum of $0$ and the sum over $i$ of  the values $\ord t_i(x)$ where $i$ is such that $0<|t_i(x)|\leq 1$. Now let $f(x)$ be $q_F^{-\alpha(x)}$ for $x\in X$.  
Then clearly $h= f(x)g(x)$ is as required.
\end{proof}

The following result about local constancy is more simple than Theorem \ref{thm:gen.sm}.
\begin{prop}[\cite{CHallp}]\label{prop:loc:const}
Let $X$ be an $\cL$-definable set and let $g$ be in $\cCexpL(X)$. Then there is a finite partition of $X$  into $\cL$-manifolds $D_i$ such that the restriction of $g$ to $D_i$ is locally constant for each $i$.
\end{prop}
\begin{proof}
Up to partitioning $X$ into finitely many $\cL$-manifolds by \cite[Proposition 1.5.3]{CCL} and restricting $g$ to the pieces, we may suppose that $X$ is an $\cL$-manifold. By Theorem \ref{thm:gen.sm} it follows that $g$ is locally constant on the complement of an $\cL$-definable set $D \subset X$ of dimension less than $\dim X$. (Alternatively this follows by the more basic theorem 4.4.3 of \cite{CHallp} and its proof.) Up to making $D$ larger if necessary, we may suppose that $X\smallsetminus D$ is an $\cL$-manifold, by the stratification result \cite[Proposition 1.5.3]{CCL}. 
Now we can finish by induction on $\dim X$ by working with the restriction of $g$ to $D$.
\end{proof}

Propositions \ref{prop:int-loc} and \ref{prop:loc:const} allow us to prove our result on loci (Theorem \ref{prop:conv}).
\begin{proof}[Proof of Theorem \ref{prop:conv}]
Let $X$ and $g$ satisfy the assumptions of the theorem, with $X\subset F^n$. By Proposition \ref{prop:int-loc}, and up to replacing $g$ without changing its zero locus, we may suppose that the complex norm $|g(x)|_\CC$  is bounded on $X$. Apply Proposition \ref{prop:loc:const} to $g$ to find $\cL$-manifolds $D_i$. For each $i$, let $\mu_{D_i}$ be the canonical measure on $D_i$ coming from the submanifold structure $D_i\subset F^n$, see Section 2.3 of \cite{CHLR}.  Note that $\mu_{D_i}$ gives a distribution on $X$.
Now $\xi := \sum_i g\mu_{D_i}$ is as desired.
\end{proof}

\section{Uniformity in the local field and in definable families}\label{sec:unif}

In this section, $F$ is no longer fixed and is no longer assumed to be of characteristic zero. On the contrary, we focus on uniformity over all local fields with as only restriction that, if $F$ has positive characteristic, then $F$ is assumed to have characteristic at least $M$ for some $M$ which may become bigger when needed. Until the end of the paper, we use terminology and notation from Section 3.1 of \cite{CHLR},  without recalling that section in full. In particular this fixes uniform notions of
\begin{itemize}
\item functions of $\cCexp$-class,
\item definable sets, and,
\item definable functions,
\end{itemize}
 where uniformity is in all local fields $F$ (with structure from the generalized Denef-Pas language) of characteristic zero and of positive characteristic at least $M$ for some $M$, denoted together by $\Loc_M$. Furthermore,  $\Loc_M'$ denotes the collection of pairs $(F,\psi)$ of  $F$ in $\Loc_M$ and $\psi$ an additive character on $F$ which is trivial on $\cM_F$ and nontrivial on $\cO_F$. A definable set $X$ is now a collection $(X_F)_{F\in \Loc_M}$ for some $M$, and, $\VF$ stands for the definable set $(F)_{F\in \Loc_M}$.

We use the following notion of $\cCexp$-families of distributions.
\begin{defn}[$\cCexp$-families of distributions] Let $Y$ and $X\subset Y\times \VF^n$ be definable sets. Suppose that for each $F\in\Loc'_M$ for some $M>0$ and each $y\in Y_F$, the set $X_{F,y}\subset F^n$ is an analytic manifold, where $X_{F,y} = \{x\in F^n\mid (y,x)\in X\}$. Let for each $F\in \Loc'_M$ and each $y\in Y_F$ a distribution $\xi_{F,y}$ be given on $X_{F,y}$. Then we call the collection of distributions $\xi_{F,y}$ for $F\in \Loc'_M$ and $y\in Y_F$ a $\cCexp$-family of distributions on $X$ over $Y$ if the collection of $B$-functions $D_{\xi_{F,y}}$ is of $\cCexp$-class (namely, there is a $\cCexp$-class function $B:X \times \ZZ\to \CC$ such that $D_{\xi_{F,y}} ( x , r )  = B_{F} (y,x,r)$ for each $F\in \Loc'_{M'}$ for some $M'\geq M$ and each $(y,x,r)\in X_F\times \ZZ$).

Denote by $\cS^*_{\cCexp / Y}(X)$ the space of $\cCexp$-families of distributions on $X$ over $Y$.

Denote by $\cS_{\cCexp / Y}(X)$ the space of $\cCexp$-functions $\varphi$ on $X$ such that $\varphi_{F}(y,\cdot):X_{F,y}\to\CC$ is a Schwartz-Bruhat function on $X_{F,y}$ for each $F \in \Loc'_{M}$ for some $M$ and each $y\in Y_F$.
\end{defn}

The following result generalizes Theorem \ref{thm:main} to uniformity in the local field $F$, and, in definable families with parameter $y$ in a definable set $Y$.  It gives two things: the uniform description of witnesses of the WF-holonomicity (the $W_i$), and, algebraic holonomicity also for local fields $F$ of positive characteristic larger than some $M$.

\begin{theorem}[Uniform holonomicity]\label{thm:main:fam} Let $Y$ and $X\subset Y\times \VF^n$ be definable sets. Suppose that $X_{F,y}$  is an analytic manifold for each $F\in \Loc'_M$ and each $y\in Y_F$.
Let $\xi$ be in $\cS^*_{\cCexp / Y}(X)$.  
Then there exist $M'$ and finitely many definable sets $W_i\subset X$, such that $W_{i,F,y}$  is an analytic submanifold of $X_{F,y}$ for each $F\in \Loc'_{M'}$, each $i$, and each $y\in Y_F$, and such that the wave front set of $\xi_{F,y}$ is contained in
$$
\cup_i CN_{W_{i,F,y}}^{X_{F,y}}.
$$
Hence, $\xi_{F,y}$ is algebraically WF-holonomic for each $F$ in $\Loc'_{M'}$ (see definition \ref{def:WF-hol}).
\end{theorem}
\begin{proof}
This follows from our holonomicity result (Theorem \ref{thm:main}), Section 4 of \cite{CHLR}, and Remark \ref{rem:amend:4.3.1} used similarly as in the Proof of Theorem \ref{thm:main} for $\cL=\cL^F$ .  
\end{proof}
For $W_{i,F}$ as in Theorem \ref{thm:main:fam} but with $Y=\{0\}$, see the appendix of \cite{CGH5} to get the extra information that the Zariski closures of the $W_{i,F}$ in $\AA_F^n$ are defined over a number field $F_0$ independently from $F$, in the case that the initial data are defined over $F_0$ as well.



The following is a family version, uniform in the local field and in definable families, of our regularization Theorem \ref{thm:reg:F:one}. 

\begin{theorem}[Uniform regularization]\label{thm:reg:fam}
Let $Y$ and $U\subset X\subset Y\times \VF^n$ be definable sets. Suppose that $X_{F,y}$  is an analytic submanifold of $F^n$ and that $U_{F,y}$ is a nonempty open of $X_{F,y}$ for each $F\in \Loc'_M$ for some $M$ and each $y\in Y_F$.
Then the restriction map
$$
\cS^*_{\cCexp/Y}(X)\to \cS^*_{\cCexp/Y}(U)
$$
admits a  linear section (in particular it is onto).
\end{theorem}

Let us also adapt Proposition \ref{thm:sec:F:one} to the uniform setting.
\begin{prop}\label{thm:sec:fam}
Let $Y$ and $Z\subset X\subset Y\times \VF^n$ be definable sets. Suppose that $X_{F,y}$ is an analytic manifold and that $Z_{F,y}$ is a closed analytic submanifold of $X_{F,y}$ for each $F\in \Loc'_M$ and each $y\in Y_F$, of lower dimension than the dimension of $X_{F,y}$.
Then the restriction map
$$
\cS_{\cCexp/Y}(X)\to \mathcal \cS_{\cCexp/Y}(Z)
$$
admits a linear section. 
\end{prop}

We also give the following, uniform (partial) converse to Theorem 3.4.1 \cite{CHLR}.

\begin{theorem}[Uniform correspondence of loci]\label{prop:conv:unif}
Let $Y$ and $X\subset Y\times \VF^n$ be definable sets. Suppose that $X_{F,y}$ is an analytic manifold for each $F\in \Loc'_M$ and each $y\in Y_F$.
Let $g$ be in $\cCexp(X)$. Suppose that for each $F\in \Loc'_M$ and each $y\in Y_F$ the zero locus of $g_{F,y}$ is dense open in $X_{F,y}$.
Then there exist $M'$ and $\xi$ in $\cS^*_{\cCexp / Y}(X)$ such that for each $F\in \Loc'_{M'}$ and each $y\in Y_F$ the zero locus of $g_{F,y}$ equals the smooth locus of $\xi_{F,y}$.
\end{theorem}

Most proofs above directly apply to the uniform setting. Let us show how to adapt the statement of Lemma \ref{lem:imp}. It is important  
that the occurring sets in the covers form moreover a definable family, in order to generalize the proof techniques for fixed $F$ above to our uniform setting. Recall from Section 3.1 of \cite{CHLR} that $\RF_{N,F}$ stands for the finite ring $\cO_F/N\cM_F$ for $F$ a local field.

\begin{lem}\label{lem:unif:imp}\label{lem:cov-4.3.1}
Let $Y$ and $X\subset Y\times \VF^n$ be definable sets. Suppose that $X_{F,y}$ is an analytic submanifold of $F^n$ of dimension $m$ for each $F\in \Loc'_M$ and each $y\in Y_F$. Then there is $N$ and a definable bijection  $\sigma: X \to \sigma(X) \subset X\times \RF_N^k$ for some $k$ which makes a commutative diagram with the projection $p:\sigma(X) \to X$ and such that each nonempty fiber of the projection $\sigma(X)\to Y\times \RF_N^k$ is an open of $X$ which equals the graph from an open subset $V_i \subset L_i$ of a linear subspace of $\VF^n$ of dimension $m$ to a linear complement of this subspace.
\end{lem}
\begin{proof}
This follows by a reasoning as for Lemma \ref{lem:imp}. 
\end{proof}

\begin{proof}[Proofs of Theorem \ref{thm:reg:fam} and proposition \ref{thm:sec:fam}]
Clearly the proofs of Theorems \ref{thm:reg:F:one}  and \ref{thm:sec:F:one} work uniformly in $F$ and in $y\in Y_F$, using Lemma \ref{lem:unif:imp} instead of Lemma \ref{lem:imp}.
\end{proof}

\begin{proof}[Proofs of Theorem \ref{prop:conv:unif}]
The proof of Theorem \ref{prop:conv} and (the statement and the proof of) Proposition \ref{prop:int-loc} adapt naturally to the uniform case.
\end{proof}

%
%
%
\begin{remark}\label{rem:hensel}
We leave it to the reader to implement the uniform results also in the analytic framework, using Remark 4.3.3 of \cite{CHLR} and our results above for $\cLan$ for fixed $F$. Similarly we let the (definable) strict $C^1$ analogues to the reader. Let us also note that within the axiomatic framework of hensel-minimality from \cite{CHR} for languages on local fields, one can investigate a  generalization of the results of this paper to that framework (with definable strict $C^1$ manifolds, sets and functions throughout).
\end{remark}


\bibliographystyle{amsplain}
\bibliography{anbib}
\end{document}